\documentclass[a4paper, 10pt]{amsart}

\usepackage{amsmath,amssymb, hyperref, color}

\newtheorem{theorem}{Theorem}[section]
\newtheorem{lemma}[theorem]{Lemma}

\newtheorem{proposition}[theorem]{Proposition}

\theoremstyle{definition}

\newcommand{\N}{\mathbb N}

 \DeclareMathOperator{\ord}{ord}

 \DeclareMathOperator{\supp}{supp}

\renewcommand{\t}{\, | \,}

\newcommand{\be}{\begin{equation}}
\newcommand{\ee}{\end{equation}}

\newcommand{\bnml}{\begin{multline}}
\newcommand{\enml}{\end{multline}}
\newcommand{\buml}{\begin{multline*}}
\newcommand{\euml}{\end{multline*}}

\newcommand{\ber}{\begin{eqnarray}}
\newcommand{\eer}{\end{eqnarray}}

\numberwithin{equation}{section}

\begin{document}

\title{A Characterization of class groups  via sets of lengths {II}}

\address{University of Graz, NAWI Graz \\
Institute for Mathematics and Scientific Computing \\
Heinrichstra{\ss}e 36\\
8010 Graz, Austria}

\email{alfred.geroldinger@uni-graz.at, qinghai.zhong@uni-graz.at}

\author{Alfred Geroldinger  and Qinghai Zhong}

\thanks{This work was supported by
the Austrian Science Fund FWF, Project Number M1641-N26.}

\keywords{Krull monoids,   maximal orders, seminormal orders; class groups,   arithmetical characterizations, sets of lengths, zero-sum sequences, Davenport constant}

\subjclass[2010]{11B30, 11R27, 13A05, 13F05, 20M13}

\begin{abstract}
Let $H$ be a Krull monoid with finite class group $G$ and suppose that every class contains a prime divisor. If an element $a \in H$ has a factorization $a=u_1 \cdot \ldots \cdot u_k$ into irreducible elements $u_1, \ldots, u_k \in H$, then $k$ is called the length of the factorization and the set $\mathsf L (a)$ of all possible factorization lengths is the set of lengths of $a$. It is classical that the system $\mathcal L (H) = \{ \mathsf L (a) \mid a \in H \}$ of all sets of lengths depends only on the class group $G$, and a standing conjecture states that conversely the system $\mathcal L (H)$ is characteristic for the class group. We verify the conjecture if the class group is isomorphic to $C_n^r$ with $r,n \ge 2$ and $r \le \max \{2, (n+2)/6\}$. Indeed, let $H'$ be a further Krull monoid with class group $G'$ such that every class contains a prime divisor and suppose that $\mathcal L (H)= \mathcal L (H')$. We prove that, if one of the groups $G$ and $G'$ is isomorphic to $C_n^r$ with $r,n$ as above, then $G$ and $G'$ are isomorphic (apart from two well-known pairings).
\end{abstract}

\maketitle

\medskip
\section{Introduction and Main Result} \label{1}
\medskip

Let $H$ be a Krull monoid with class group $G$ and suppose that every class contains a prime divisor (holomorphy rings in global fields are such Krull monoids and more examples will be given in Section \ref{2}). Then every nonunit $a \in H$ can be written as a product of irreducible elements, say $a = u_1 \cdot \ldots \cdot u_k$, and the number of factors $k$ is called the length of the factorization. The set $\mathsf L (a)$ of all possible factorization lengths is the set of lengths of $a$, and $\mathcal L (H) = \{ \mathsf L (a) \mid a \in H \}$ is called the system of sets of lengths of $H$ (for convenience we set $\mathsf L (a) = \{0\}$ if $a$ is an invertible element of $H$). It is easy to check that all sets of lengths are finite and, by definition of the class group, we observe that
$H$ is factorial if and only if $|G|=1$. By a result due to  Carlitz in 1960, we know that $H$ is half-factorial (i.e., $|L|=1$ for all $L \in \mathcal L (H)$) if and only if $|G| \le 2$.

Suppose that $|G| \ge 3$. Then there is some $a \in H$ with $|\mathsf L (a)|>1$. If $k, \ell \in \mathsf L (a)$ with $k < \ell$ and $m \in \N$, then $\mathsf L (a^m) \supset \{km + \nu (\ell-k) \mid \nu \in [0,m] \}$ which shows that sets of lengths can become arbitrarily large.
 The monoid $\mathcal B (G)$ of zero-sum sequences over $G$ is again a Krull monoid with class group isomorphic to $G$, every class contains a prime divisors, and the systems of sets of lengths of $H$ and that of $\mathcal B (G)$ coincide. Thus $\mathcal L (H) = \mathcal L \big(B (G) \big)$, and it is usual to set $\mathcal L (G) := \mathcal L \big( \mathcal B (G) \big)$. In particular, the system of sets of lengths of $H$ depends only on the class group $G$. The associated inverse question asks whether or not sets of lengths are characteristic for the class group.
More precisely, the Characterization Problem for class groups can be formulated as follows (see \cite[Section 7.3]{Ge-HK06a}, \cite[page 42]{Ge-Ru09}, \cite{Sc09b}, and Proposition \ref{2.1})).

\begin{enumerate}
\item[]
Given two finite abelian groups $G$ and $G'$ such that $\mathcal L(G) = \mathcal L(G')$.
Does it follow that $G \cong G'$?
\end{enumerate}

The system of sets of lengths $\mathcal L (G)$ is studied with methods from additive combinatorics. In particular, zero-sum theoretical invariants (such as the Davenport constant or the cross number) and the associated inverse problems play a crucial role.
Most of these invariants are well-understood only in a very limited number of cases (e.g., for groups of rank two, the precise value of the Davenport constant $\mathsf D (G)$ is known and the associated inverse problem is solved; however, if $n$ is not a prime power and $r \ge 3$, then the value of the Davenport constant $\mathsf D (C_n^r)$ is unknown). Thus it is not surprising that affirmative answers to the Characterization Problem so far have been restricted to those groups where we have a good understanding of the Davenport constant. These groups include elementary $2$-groups, cyclic groups, and groups of rank two  (the latter were recently handled in \cite{Ge-Sc15b}; for a variety of partial results we refer to \cite{Sc05d,Sc11b, Sc09c}).

The goal of the present note is to solve the Characterization Problem for groups of the form $C_n^r$ if the exponent is large with respect to the rank. Here is our main theorem.

\smallskip
\begin{theorem} \label{1.1}
Let $G$ be an abelian group such that $\mathcal L (G) = \mathcal L ( C_n^r)$ where $r, n \in \N$ with $n \ge 2$, $(n,r)\notin \{(2,1),(2,2), (3,1)\}$, and $r \le \max \{2, (n+2)/6 \}$. Then $G \cong C_n^r$.
\end{theorem}

The groups $C_n^r$, where $r, n$ are as above, are the first groups at all for which the Characterization Problem is solved whereas the Davenport constant is unknown. This is made possible by a detailed study of the set of minimal distances $\Delta^* (G) = \{ \min \Delta (G_0) \mid G_0 \subset G \ \text{is a non-half-factorial subset} \}$
and the associated minimal non-half-factorial subsets. Sets of minimal distances have been investigated by Chapman, Grynkiewicz, Hamidoune, Plagne, Schmid,  Smith, and others (see \cite[Section 6.8]{Ge-HK06a} for some basic information and \cite{Ge-Ha02, Pl-Sc05b, Ch-Sc-Sm08b, Sc09c, B-G-G-P13a, Ge-Zh15a, Pl-Sc16a} for recent progress).
 In Section \ref{2} we repeat some key facts on Krull monoids and gather the required machinery, and in Section \ref{3} we study structural properties of (large) minimal non-half-factorial sets.
 The proof of Theorem \ref{1.1}   will be provided in Section \ref{4} where we also give a positive answer to the Characterization Problem for all groups $G$ with Davenport constant $\mathsf D (G) \in [4,11]$ (Proposition \ref{4.1}).

\medskip
\section{Background on Krull monoids and their sets of minimal distances} \label{2}
\medskip

Our notation and terminology are consistent with \cite{Ge-Sc16a}. We denote by $\mathbb N$ the set of positive integers, and for $a, b \in \mathbb Q$, we denote
by $[a, b ] = \{ x \in \mathbb Z \mid a \le x \le b\}$ the discrete, finite interval between $a$ and $b$. If $A, B \subset \mathbb Z$ are subsets of the integers, then $A+B = \{a+b \mid a \in A, b \in B\}$ denotes their {\it sumset}, and  $\Delta (A)$ the {\it set of $($successive$)$ distances} of $A$ (that is, $d \in \Delta (A)$ if and only if $d = b-a$ with $a, b \in A$ distinct and $[a, b] \cap A = \{a, b\}$). Let \ $d, l  \in \mathbb N$ \ and \ $M \in \mathbb N_0$. A subset \ $L
\subset \mathbb Z$ \ is called an {\it almost arithmetical
progression} ({\rm AAP} for short) with {\it difference} $d$, {\it length} $l$, and
{\it bound} $M$ if
\[
L = y + (L' \cup L^* \cup L'') \subset y + d \mathbb Z,
\]
where $y \in \mathbb Z$, $L^* = \{ \nu d \mid \nu \in [0, l] \}$ is
an arithmetical progression with difference $d$ and length $l$, $L'
\subset [-M,-1]$, and $L'' \subset \max L^* + [1, M]$.

By a monoid we mean a commutative semigroup with identity which satisfies the cancellation laws.
A monoid $F$ is called  free abelian with basis $P \subset F$, and we write $F = \mathcal F (P)$, if every $a \in F$ has a unique representation of the form
\[
a = \prod_{p \in P} p^{\mathsf v_p(a) } \quad \text{with} \quad
\mathsf v_p(a) \in \N_0 \ \text{ and } \ \mathsf v_p(a) = 0 \ \text{
for almost all } \ p \in P \,.
\]
A monoid $H$ is said to be a {\it Krull monoid} if it satisfies one of the following two equivalent conditions (see \cite[Theorem 2.4.8]{Ge-HK06a}).
\begin{itemize}
\item[(a)] $H$ is $v$-noetherian and completely integrally closed.

\item[(b)] There exists a monoid homomorphism $\varphi \colon H \to F = \mathcal F (P)$ into a free abelian monoid $F$ such that $a \t b$ in $H$ if and only if $\varphi (a) \t \varphi (b)$ in $F$.
\end{itemize}
 Rings of integers, holomorphy rings in algebraic function fields, and regular congruence monoids in these domains are Krull monoids with finite class group such that every class contains a prime divisor (\cite[Section 2.11 and Examples 7.4.2]{Ge-HK06a}). Monoid domains and power series domains that are Krull are discussed in \cite{Ki-Pa01, Ch11a}, and note that every class of a  Krull monoid domain contains a prime divisor. For monoids of modules that are Krull and their distribution of prime divisors, we refer the reader to \cite{Fa06a, Ba-Ge14b}.

Sets of lengths in Krull monoids can be studied in the monoid of zero-sum sequences over its class group. Let $G$ be an additively written abelian group and $G_0 \subset G$ a subset. An element
\[
S = g_1 \cdot \ldots \cdot g_{\ell} = \prod_{g \in G_0} g^{\mathsf v_g (S)} \in \mathcal F (G_0)
\]
 is called a sequence over $G_0$, and we use all notations as in \cite{Gr13a}. In particular,  $\sigma (S) = g_1+ \ldots + g_{\ell}$
denotes the sum,  $|S| = \ell$ the length, $\mathsf h (S) = \max \{ \mathsf v_g (S) \mid g \in G_0 \}$  the maximal multiplicity, $\supp (S) = \{g_1, \ldots, g_{\ell}\} \subset G_0$ the support, and $\mathsf k (S) = \sum_{i=1}^{\ell} 1/\ord (g_i)$ the cross number of $S$. The monoid
\[
\mathcal B (G_0) = \{ S \in \mathcal F (G_0) \mid \sigma (S) = 0 \}
\]
is the monoid of zero-sum sequences over $G_0$, and since the embedding $\mathcal B (G_0) \hookrightarrow \mathcal F (G_0)$ satisfies Condition (b) above, $\mathcal B (G_0)$ is a Krull monoid. As usual, we write $\mathcal L (G_0) = \mathcal L \big( \mathcal B (G_0) \big)$ for the system of sets of lengths of $\mathcal B (G_0)$ and $\mathcal A (G_0) = \mathcal A (\mathcal B (G_0))$ for the set of atoms (the set of irreducible elements) of $\mathcal B (G_0)$. Note that the atoms of $\mathcal B (G_0)$ are precisely the minimal zero-sum sequences over $G_0$, and
\[
\mathsf D (G_0) = \sup \{ |U| \mid U \in \mathcal A (G_0) \} \in \mathbb N \cup \{\infty\}
\]
is the {\it Davenport constant} of $G_0$. The significance of the system of sets of lengths $\mathcal L (G)$ (and hence of the Characterization Problem in the formulation given in the Introduction) stems from its universal role which can be seen from the following proposition.

\medskip
\begin{proposition} \label{2.1}~

\begin{enumerate}
\item If $H$ is a Krull monoid with class group $G$ such that each class contains a prime divisor, then $\mathcal L (H) = \mathcal L  (G) $.

\item Let $\mathcal O$ be a holomorphy ring in a global field $K$, $A$ a central simple algebra over $K$, and $H$ a classical maximal $\mathcal O$-order of $A$ such that every stably free left $R$-ideal is free. Then $\mathcal L (H) = \mathcal L      (G) $, where  $G$ is a ray class group of $\mathcal O$ and hence finite abelian.

\item Let $H$ be a seminormal order in a holomorphy ring of a global field with principal order $\widehat H$ such that the natural map $\mathfrak X (\widehat H) \to \mathfrak X (H)$ is bijective and there is an isomorphism $\overline{\vartheta}\colon \mathcal{C}_v(H)\rightarrow \mathcal{C}_v(\widehat{H})$ between the $v$-class groups. Then $\mathcal L (H) = \mathcal L     (G)$, where $G = \mathcal{C}_v(H)$ is  finite abelian.
\end{enumerate}
\end{proposition}

\begin{proof}
1. See \cite[Section 3.4]{Ge-HK06a}.

2. See \cite[Theorem 1.1]{Sm13a}, and \cite{Ba-Sm15} for related results of this flavor.

3. See \cite[Theorem 5.8]{Ge-Ka-Re15a} for a more general result in the setting of weakly Krull monoids.
\end{proof}

\smallskip
Next we discuss sets of distances and minimal sets of distances. Let
\[
\Delta (G_0) = \bigcup_{L \in \mathcal L (G_0)} \Delta (L) \ \subset \N
\]
denote the {\it set of distances} of $G_0$. Then $G_0$ is called half-factorial if $\Delta (G_0) = \emptyset$. Otherwise, $G_0$ is called non-half-factorial and we have $\min \Delta (G_0) = \gcd \Delta (G_0)$. Note that $G_0$ is half-factorial if and only if $\mathsf k (A)=1$ for all $A \in \mathcal A (G_0)$. Furthermore, the set $G_0$ is called
\begin{itemize}
\item minimal non-half-factorial if it is half-factorial and every proper subset $G_1 \subsetneq G_0$ is half-factorial.

\item an LCN-set if $\mathsf k (A) \ge 1$ for all $A \in \mathcal A (G_0)$.
\end{itemize}
The set $\Delta (G)$ is an interval and its maximum is studied in \cite{Ge-Gr-Sc11a}. The following two subsets of $\Delta (G)$, the {\it set of minimal distances} $\Delta^* (G)$ and the set $\Delta_1 (G)$, play a crucial role in the present paper. We define    \[
\begin{aligned}
\Delta^* (G) & = \{\min \Delta (G_0) \mid G_0 \subset G \ \text{with} \ \Delta (G_0) \ne \emptyset \} \subset \Delta (G) \,, \\
\mathsf m (G) & = \max \{ \min \Delta (G_0) \mid G_0 \subset G \ \text{is an LCN-set with} \ \Delta (G_0) \ne \emptyset \} \,,
\end{aligned}
\]
and we denote by $\Delta_1 (G)$ the set of all $d \in \N$ with the following property:
\begin{itemize}
\item[]  For every $k \in \N$, there exists some $L \in \mathcal L (G)$  which is an AAP with difference $d$ and length $l \ge k$.
\end{itemize}
Thus, by definition, if $G'$ is a further finite abelian group such that $\mathcal L (G) = \mathcal L (G')$, then $\Delta_1 (G) = \Delta_1 (G')$.
The next proposition gathers the  properties of $\Delta^* (G)$ and of $\Delta_1 (G)$ which are needed in the sequel.

\medskip
\begin{proposition} \label{2.2}
Let $G$ be a finite abelian group with $|G|\ge 3$ and $\exp (G)=n$.
\begin{enumerate}
\item $\Delta^* (G) \subset \Delta_1 (G) \subset \{ d_1 \in \Delta (G) \mid d_1 \ \text{divides some } \ d \in \Delta^* (G)\}$. In particular, $\max \Delta^* (G) = \max \Delta_1 (G)$.

\smallskip
\item $\max \Delta^* (G) = \max \{ \exp (G)-2, \mathsf m (G) \} = \max \{ \exp (G)-2, \mathsf r (G)-1\}$. If $G$ is a $p$-group, then $\mathsf m (G)=\mathsf r (G)-1$.

\smallskip
\item If $k\in \N$ is maximal such that $G$ has a subgroup isomorphic to $C_{n}^k$, then
      \[
      \Delta_1(G)\subset [1,\max \{ \mathsf m(G), \lfloor \frac{n}{2}\rfloor-1 \}]  \cup [\max \{ 1,n-k-1 \} ,n-2]\,.
      \]
      and
      \[
      [1,\mathsf r(G)-1]  \cup  [\max \{ 1,n-k-1\} ,n-2] \subset \Delta_1(G) \,.
      \]
\end{enumerate}
\end{proposition}

\begin{proof}
1. follows from \cite[Corollary 4.3.16]{Ge-HK06a} and 2. from \cite[Theorem 1.1 and Proposition 3.2]{Ge-Zh15a}.

3. In  \cite[Theorem 3.2]{Sc09c}, it is proved that $\Delta^* (G)$ is contained in the set given above. Since this set contains all its divisors, $\Delta_1 (G)$ is contained in it by 1. The set $[1, \mathsf r (G)-1] \cup  [\max \{ 1,n-k-1\} ,n-2]$ is contained in $\Delta_1 (G)$ by \cite[Propositions 4.1.2 and 6.8.2]{Ge-HK06a}.
\end{proof}

\medskip
\section{Minimal non-half-factorial subsets} \label{3}
\medskip

\centerline{\it Throughout this section, let $G$ be an additive finite abelian group}
\centerline{\it  with $|G| \ge 3$, $\exp (G)=n$, and $\mathsf r (G)=r$.}

\medskip
The following two technical lemmas will be used throughout the manuscript.

\medskip
\begin{lemma} \label{3.1}
Let  $G_0 \subset G$  a subset.
\begin{enumerate}
\item For each $g \in G_0$,
      \[
      \begin{aligned}
      &\gcd \big( \{ \mathsf v_g (B) \mid B \in \mathcal B (G_0) \} \big)  =  \gcd \big( \{ \mathsf v_g (A) \mid A \in \mathcal A (G_0) \} \big)\\ =&
      \min  \big( \{ \mathsf v_g (A) \mid \mathsf v_g (A)>0,  A  \in \mathcal A (G_0) \} \big)   =  \min \big( \{ \mathsf v_g (B) \mid \mathsf v_g (B) > 0, B \in \mathcal B (G_0) \} \big) \\ =& \min  \big( \{ k \in \N  \mid kg \in \langle G_0 \setminus \{g\} \rangle  \} \big)=\gcd \big(  \{ k \in \N  \mid kg \in \langle G_0 \setminus \{g\} \rangle  \} \big) \,.
      \end{aligned}
      \]
In particular, $\min  \big( \{ k \in \N  \mid kg \in \langle G_0 \setminus \{g\} \rangle  \} \big)$ divides $\ord (g)$.

\smallskip
\item Suppose that for any $h\in G_0$, we have that  $h\not\in \langle G_0\setminus \{h,\ h'\}\rangle $ for any $h'\in G_0\setminus \{h\}$.
 Then for any atom $A$ with $\supp(A)\subsetneq G_0$ and any $h \in \supp (A)$, we have $\gcd (\mathsf v_h(A), \ord(h) )>1$.

\smallskip
\item If $G_0$ is minimal non-half-factorial, then there exists a
   minimal  non-half-factorial subset $G_0^* \subset G$ with $|G_0|=|G_0^*|$ and a transfer homomorphism $\theta \colon \mathcal B (G_0) \to \mathcal B (G_0^*)$ such that  the following  properties are satisfied{\rm \,:}
    \begin{enumerate}
    \item   For each $g \in G_0^*$, we have $g \in \langle G_0^* \setminus \{g\} \rangle$.

    \smallskip
    \item  For each $B \in \mathcal B (G_0)$, we have $\mathsf k (B) = \mathsf k \big( \theta (B) \big)$.
  \smallskip
  \item If $G_0^*$ has the property that  for each $h\in G_0^*$, $h\not\in \langle E\rangle$ for any $E\subsetneq G_0^*\setminus \{h\}$, then $G_0$ also has the property.
    \end{enumerate}
\end{enumerate}
\end{lemma}

\begin{proof}
See \cite[Lemma 3.4]{Ge-Zh15a}.
\end{proof}

\medskip
\begin{lemma} \label{3.2}~

\begin{enumerate}
\item If $g \in G$ with $\ord (g) \ge 3$, then $\ord (g) - 2 \in \Delta^* (G)$. In particular, $n-2 \in \Delta^* (G)$.

\smallskip
\item If $r \ge 2$, then $[1, \mathsf r -1] \subset \Delta^* (G)$.

\smallskip
\item Let $G_0 \subset G$ a subset.
      \begin{enumerate}
      \item If there exists an $U \in \mathcal A (G_0)$ with $\mathsf k (U) < 1$, then $\min \Delta (G_0) \le \exp (G)-2$.

      \smallskip
      \item If $G_0$ is an \text{\rm LCN}-set, then $\min \Delta (G_0) \le |G_0|-2$.
      \end{enumerate}
\end{enumerate}
\end{lemma}

\begin{proof}
See \cite[Proposition 6.8.2 and Lemmas 6.8.5 and 6.8.6]{Ge-HK06a}.
\end{proof}

\medskip
\begin{lemma} \label{3.3}
Let  $G_0\subset G$ be a subset,  $g\in G_0$, and $s$ the smallest integer such that $sg\in \langle G_0\setminus\{g\}\rangle$, and suppose that $s<\ord(g)$.
Then $\ord(sg)>1$ and  for each prime $p$ dividing $\ord(sg)$, there exists an atom $A\in \mathcal A(G_0)$ with $2\le |\supp(A)|\le \mathsf r(G)+1$, $s \le  \mathsf v_g(A) \le  \ord(g)/2$, and $p\nmid \frac{\mathsf v_g(A)}{s}$. In particular,
\begin{enumerate}
\item If $|G_0|\ge \mathsf r(G)+2$, then there exist $s_0<\ord(g)$ and $E\subsetneq G_0\setminus \{g\} $ such that $s_0g\in \langle E \rangle$.

\item If $s=1$ and $\ord(g)$ is a prime power, then there exists a subset $E\subset G_0\setminus\{g\}$ with $|E|\le \mathsf r(G)$ such that $g\in \langle E\rangle $.
\end{enumerate}
\end{lemma}

\begin{proof}
 We set $\exp(G)=n = p_1^{k_1} \cdot \ldots \cdot p_t^{k_t}$, where $t, k_1, \ldots, k_t \in \N$ and $p_1, \ldots, p_t$ are distinct primes. Since $s<\ord(g)$, we have that $\ord(sg)>1$.
  Let $\nu \in [1,t]$ with $p_{\nu} \t \ord (sg)$.
Since $sg\in \langle G_0\setminus \{g\} \rangle$, it follows that $0 \ne \frac{n}{p_{\nu}^{k_{\nu}}}sg\in G_{\nu} = \langle \frac{n}{p_{\nu}^{k_{\nu}}}h \mid h\in G_0\setminus \{g\}\rangle$. Obviously, $G_{\nu}$ is a $p_{\nu}$-group.
Let $E_{\nu} \subset G_0\setminus \{g\}$ be minimal such that $\frac{n}{p_{\nu}^{k_{\nu}}}sg \in  \langle \frac{n}{p_{\nu}^{k_{\nu}}} E_{\nu}\rangle$. Since $\langle \frac{n}{p_{\nu}^{k_{\nu}}} E_{\nu}\rangle \subset G_{\nu}$ and $G_{\nu}$ is a $p_{\nu}$-group,
   it follows that
\[
1\le |E_{\nu}| = |\frac{n}{p_{\nu}^{k_{\nu}}} E_{\nu}| \le \mathsf r (G_{\nu})  \le \mathsf r (G) \,.
\]
Let $d_{\nu} \in \N$ be  minimal  such that $d_{\nu}g\in \langle E_{\nu}\rangle$. Since
$0 \ne \frac{n}{p_{\nu}^{k_{\nu}}}sg\in \langle E_{\nu} \rangle$, it follows that $d_{\nu} < \ord (g)$.
By Lemma \ref{3.1}.1, $d_{\nu}\t \gcd(\frac{n}{p_{\nu}^{k_{\nu}}}s, \ord(g))$ and there exists an atom $U_{\nu}$ such that $\mathsf v_g(U_{\nu})=d_{\nu}$ and $|\supp(U_{\nu})\setminus\{g\}| \le |E_{\nu}| \le \mathsf r(G)$. Since $\mathsf v_g (U_{\nu}) = d_{\nu} < \ord (g)$, it follows that $|\supp (U_{\nu})| \ge 2$. By the minimality of $s$ and $d_{\nu}\t \frac{n}{p_{\nu}^{k_{\nu}}}s$, we have that $s\t d_{\nu}$ and $p_{\nu}\nmid \frac{d_{\nu}}{s}$.

If $|G_0|\ge \mathsf r(G)+2$, then $|E_{\nu}|\le \mathsf r(G)<|G_0\setminus \{g\}|$ implies that $E_{\nu} \subsetneq G_0 \setminus \{g\}$, and the assertion holds with $E=E_{\nu}$ and $s_0 = d_{\nu}$.

If $s=1$ and $\ord(g)$ is a prime power, then $\ord(g)$ is a power of $p_{\nu}$ which implies that  $\gcd(\frac{n}{p_{\nu}^{k_{\nu}}}s, \ord(g))=1$ whence $d_{\nu}=1$ and $g\in \langle E_{\nu}\rangle$.
\end{proof}

\medskip
\begin{lemma} \label{3.4}
Let $G_0 \subset G$ be a minimal non-half-factorial {\rm LCN}-set with $|G_0|\ge r+2$.  Suppose that
for any $h\in G_0$, $h\in \langle G_0\setminus\{h\}\rangle$ but $h\not\in \langle G_0\setminus \{h,\ h'\}\rangle $ for any $h'\in G_0\setminus \{h\}$. Then
 $|G_0|\le r+\frac{n}{2}$.
In particular, if  each atom $A\in \mathcal A(G_0)$ with $\supp(A)=G_0$ has cross number $\mathsf k(A)>1$,
then $\min \Delta(G_0)\le \frac{5n}{6}-4$.
\end{lemma}

\begin{proof}
We choose an element $g\in G_0$. If $\ord(g)$ is a prime power, then there exists $E\subset G_0\setminus \{g\}$ such that $g\in \langle E\rangle$ and $|E|\le r<|G_0|-1$ by Lemma \ref{3.3}, a contradiction to the assumption on $G_0$. Thus $\ord(g)$ is not a prime power.

Let $s \in \N$ be  minimal  such that there exists a subset $E\subsetneq G_0\setminus \{g\}$ with $sg\in \langle E \rangle$, and  by Lemma \ref{3.3}.1, we observe that $s < \ord (g)$. Let  $E \subsetneq G_0\setminus \{g\}$ be  minimal such that $sg\in \langle E \rangle$. By Lemma \ref{3.1}.1, there is an atom   $V$  with $\mathsf v_g(V)=s\t \ord(g)$ and $\supp(V)=\{g\}\cup E\subsetneq G_0$.
By Lemma \ref{3.1}.2, for each $h\in \supp(V)$, $\mathsf v_h(V)\ge 2$. Since $G_0$ is a minimal non-half-factorial LCN set, we obtain that
\[
1=\mathsf k(V)\ge  \frac{2}{n}(|E|+1) \,,
\]
whence $|E|\le \frac{n}{2}-1$.

Since $s \ge 2$, there is a  prime  $p\in \N$   dividing $s$ and hence $p \t s \t \ord (g)$.  By Lemma \ref{3.3}, there exists an atom $U_1$ such that $|\supp(U_1)|\le r+1$ and  $p\nmid  \mathsf v_g(U_1)$, and therefore $\supp(U_1)\subsetneq G_0$.

Let $d=\gcd(s,\mathsf v_g(U_1))$ and $E_1=\supp(U_1)\setminus\{g\}$.  Then $d<s$ and $dg \in \langle sg, \mathsf v_g(U_1)g \rangle \subset \langle E\cup E_1 \rangle \subset \langle G_0\setminus \{g\} \rangle$.
The minimality of $s$ implies  that $E\cup E_1=G_0\setminus \{g\}$, and thus  $|G_0|\le 1+|E|+|E_1|\le 1+ r+\frac{n}{2}-1= r+\frac{n}{2}$.

Suppose that  each atom $A\in \mathcal A(G_0)$ with $\supp(A)=G_0$ has cross number $\mathsf k(A)>1$.
There exist $x_1\in [1,\frac{\ord(g)}{s}-1]$ and $x_2\in [1,\frac{\ord(g)}{\mathsf v_g(U_1)}-1]$ such that $dg=x_1sg+x_2\mathsf v_g(U_1)g$. Thus $d+y\ord(g)=x_1s+x_2\mathsf v_g(U_1)$ with some $y\in \N_0$. Let $V^{x_1}U_1^{x_2}=(g^{\ord(g)})^y\cdot W$, where $W \in \mathcal B (G)$  with $\mathsf v_g(W)=d$, and let $W_1$ be an atom dividing $W$ with $\mathsf v_g(W_1)>0$. Since $\mathsf v_g(W_1)\le d<s$, the minimality of $s$ implies that $\supp(W_1)=G_0$ and hence  $\mathsf k(W_1)>1$. Since $G_0$ is minimal non-half-factorial, we have that $\mathsf k(V)=\mathsf k(U_1)=1$. Therefore there exists $l\in \N$ with $2\le l<x_1+x_2$ such that $\{l, x_1+x_2\}\subset \mathsf L(V^{x_1}U_1^{x_2})$. Then
\[
\min \Delta(G_0)\le x_1+x_2-l\le \frac{\ord(g)}{s}+\frac{\ord(g)}{\mathsf v_g(U_1)}-4\le \frac{5n}{6}-4 \,.  \qedhere
\]
\end{proof}

\medskip
For our next result we need the following technical lemma

\medskip
\begin{lemma}\label{3.5}
Let $G_0\subset G$ be a non-half-factorial subset satisfying the following two conditions:
\begin{itemize}
\item[(a)] There exists some $g\in G_0$ such that $\Delta (G_0\setminus \{g\})=\emptyset$.
\item[(b)] There exists some  $U\in \mathcal A(G_0)$ with $\mathsf k(U)=1$ and $\gcd(\mathsf v_g(U),\ord(g))=1$.
\end{itemize}
Then $\mathsf k(\mathcal A(G_0))\subset \N$ and
\[
\min \Delta (G_0)\mid \gcd \{\mathsf k(A)-1 \mid A\in \mathcal A(G_0)\}\,.
\]
Note that the conditions hold if $\Delta(G_1)=\emptyset$ for each $G_1\subsetneq G_0$ and there exists some $G_2$ such that $\langle G_2\rangle=\langle G_0\rangle$ and $|G_2|\le |G_0|-2$.
\end{lemma}

\begin{proof}
The first statement follows from \cite[Lemma 6.8.5]{Ge-HK06a}.  If $\Delta (G_1) = \emptyset$ for all $G_1 \subsetneq G_0$, then Condition (a) holds. Let $G_2 \subsetneq G_1 \subsetneq G_0$ with $\langle G_2 \rangle = \langle G_0 \rangle$. If $g \in G_1 \setminus G_2$, then $\langle G_2 \rangle = \langle G_0 \rangle$ implies that there is some $U \in \mathcal A (G_1)$ with $\mathsf v_g (U)=1$, and since $G_1 \subsetneq G_0$, it follows that $\mathsf k (U)=1$.
\end{proof}

\medskip
\begin{lemma} \label{3.6}
Suppose that $\exp(G)=n$ is not a prime power.
Let $G_0 \subset G$ be a minimal non-half-factorial {\rm LCN}-set with $|G_0|\ge r+2$ such that $h \in \langle G_0 \setminus \{h\} \rangle$ for every $h \in G_0$. Suppose that one of the following properties is  satisfied{\rm \,:}
\begin{enumerate}
\item[(a)] For any $h\in G_0$,   $h\not\in \langle G_0\setminus \{h,\ h'\}\rangle $ for any $h'\in G_0\setminus \{h\}$ and
           there exists an atom $A\in \mathcal A(G_0)$ with  $\mathsf k(A)=1$  and $\supp(A)=G_0$.

\smallskip
\item[(b)] There is a subset $G_2 \subset G_0$ such that $\langle G_2 \rangle = \langle G_0 \rangle$ and $|G_2| \le |G_0|-2$.
\end{enumerate}
Then $\min \Delta (G_0) \le   \frac{n+r-3}{2}$.
\end{lemma}

\begin{proof}
Assume to the contrary that $\min \Delta (G_0) \ge  \frac{n+r}{2}-1$. Then Lemma \ref{3.2}.3.(b) implies that $|G_0|\ge \frac{n+r}{2}+1$. If  Property (a) is satisfied, then Lemma \ref{3.1}.2 implies that there exists some $g\in G_0$ such that $\mathsf v_g(A)=1$.
By Lemma \ref{3.5}, each of the two Properties (a) and (b) implies that
 $\mathsf k (U) \in \N$ for each $U \in \mathcal A (G_0)$ and
\[
\min \Delta (G_0) \t \gcd \big( \{ \mathsf k (U) - 1 \mid U \in \mathcal A (G_0) \} \big) \,.
\]
We set
\[
\Omega_{=1}=\{A\in \mathcal A (G_0) \mid \mathsf k (A)=1\} \quad \text{and} \quad \Omega_{>1}=\{A\in \mathcal A (G_0) \mid \mathsf k (A)>1\} \,.
\]
Thus for each $U_1, U_2 \in \Omega_{>1}$ we have
\begin{equation} \label{e2}
\mathsf k (U_1)\ge \frac{n+r}{2} \quad \text{and} \quad \big(\text{either} \ \mathsf k (U_1)=\mathsf k (U_2) \ \text{or} \ |\mathsf k (U_1)-\mathsf k (U_2)|\ge  \frac{n+r}{2}-1 \big) \,.
\end{equation}
Furthermore, for each $U \in \Omega_{=1}$ we have $\mathsf h (U) \ge 2$ (otherwise, $U$ would divide every atom $U_1 \in \Omega_{>1}$).
We claim that

\begin{enumerate}
\item[{\bf  A1.}] For each $U\in \Omega_{>1}$,  there are $A_1,  \ldots , A_m \in \Omega_{=1}$,
                  where $m \le \frac{n+1}{2}$,  such that $UA_1 \cdot \ldots \cdot A_m$ can be factorized into a product of atoms from $\Omega_{=1}$.
\end{enumerate}

\medskip

{\it Proof of {\bf  A1.}} \, Suppose that  Property (a) holds. As observed above there exists some $g\in G_0$ such that $\mathsf v_g(A)=1$.  Lemma \ref{3.3} implies that there is an atom $X$ such that
$2\le |\supp(X)|\le \mathsf r(G)+1$ and $1 \le  \mathsf v_g(X) \le  \ord(g)/2$. Since  $g\not\in \langle G_0\setminus \{g,\ h\}\rangle$ for any $h\in G_0\setminus \{g\}$, it follows that $\mathsf v_g (X) \ge 2$, and  $|G_0|\ge  r+2$ implies  $\supp(X)\subsetneq G_0$.

Suppose that Property (b) is satisfied.
We choose an element $g\in G_0\setminus G_2$. Then $g\in \langle G_2 \rangle$ and by Lemma \ref{3.1}.1,  there is an atom $A'$ with
$\mathsf v_g(A')=1$ and $\supp(A')\subset G_2 \cup \{g\}\subsetneq G_0$. This implies that  $A'\in \Omega_{=1}$. Let $h \in G_0$ such that  $\mathsf v_h(A')=\mathsf h(A')$. Since $\mathsf h(A')\ge 2$, we obtain that
$A'^{\lceil\frac{\ord(h)}{\mathsf h(A')}\rceil}=h^{\ord(h)}\cdot W$
where $W$ is a product of $\lceil\frac{\ord(h)}{\mathsf h(A')}\rceil-1$ atoms and $\mathsf v_g(W)=\lceil\frac{\ord(h)}{\mathsf h(A')}\rceil$.
Thus there exists an atom $X'$ with $2\le\mathsf v_g(X')\le \lceil\frac{\ord(h)}{\mathsf h(A')}\rceil\le \frac{n}{2}+1$.

\smallskip
Therefore both properties imply that there are  $A, X \in \mathcal A (G_0)$ and  $g\in G_0$ such that $\mathsf k(A)=\mathsf k(X)=1$, $\mathsf v_{g}(A)=1$,  and $2\le\mathsf v_g(X)\le \frac{n}{2}+1$. Let $U\in \Omega_{>1}$.
\smallskip

If $\ord(g)-\mathsf v_g(U)< \mathsf v_g(X)\le \frac{n}{2}+1$, then
\[
U  A^{\ord(g)-\mathsf v_g(U)}=g^{\ord(g)}  S\,,
\]
where $S \in \mathcal B (G_0)$  and $\ord(g)-\mathsf v_g(U)\le \frac{n}{2}$. Since $\supp (S) \subsetneq G_0$, $S$ is a product of atoms from $\Omega_{=1}$.

If $\ord(g)-\mathsf v_g(U)\ge \mathsf v_g(X)$, then
\[
U X^{\lfloor \frac{\ord(g)-\mathsf v_g(U)}{\mathsf v_g(X)}\rfloor} A^{\ord(g)-\mathsf v_g(U)-\mathsf v_g(X)\cdot \lfloor \frac{\ord(g)-\mathsf v_g(U)}{\mathsf v_g(X)}\rfloor }=g^{\ord(g)}  S\,,
\]
where $S$ is a product of atoms from $\Omega_{=1}$ (because $\supp (S) \subsetneq G_0)$ and
\[
\begin{aligned}
\lfloor  & \frac{\ord(g)-\mathsf v_g(U)}{\mathsf v_g(X)}\rfloor+\ord(g)-\mathsf v_g(U)-\mathsf v_g(X)\cdot \lfloor \frac{\ord(g)-\mathsf v_g(U)}{\mathsf v_g(X)}\rfloor \\
& \le \frac{\big( \ord (g) - \mathsf v_g (U)\big) - \big(\mathsf v_g (X)-1 \big)}{\mathsf v_g (X)} + \mathsf v_g (X)-1 \\
& \le \frac{\ord (g)-\mathsf v_g (U) + 1}{2} \le \frac{n+1}{2}\,. \qquad \qquad \qquad  \qed{\text{\rm (Proof of {\bf  A1})}}
\end{aligned}
\]

\smallskip
We set
\[
\Omega_{>1}'=\{A\in \mathcal A (G_0) \mid \mathsf k (A)= \min \{ \mathsf k (B) \mid B \in \Omega_{>1}\} \}\subset \Omega_{>1} \,,
\]
and we consider all tuples $(U,A_1,\ldots,A_m)$, where $U\in \Omega_{>1}'$ and $A_1,\ldots,A_m\in \Omega_{=1}$,  such that $UA_1 \cdot \ldots \cdot A_m$ can be factorized into a product of atoms from $\Omega_{=1}$. We fix one such tuple $(U,A_1,\ldots,A_m)$ with the property that $m$ is minimal possible.
Let
\begin{equation}\label{e3}
UA_1 \cdot \ldots \cdot A_m = V_1 \cdot \ldots \cdot V_t \quad \text{ with} \quad t \in \N \quad \text{and} \quad  V_1, \ldots, V_t \in \Omega_{=1} \,.
\end{equation}
We observe that  $\mathsf k(U)=t-m$ and continue with the following assertion.

\begin{enumerate}
\item[{\bf  A2. }]  For each $\nu \in [1, t]$, we have $V_{\nu} \nmid UA_1 \cdot \ldots \cdot A_{m-1}$.
\end{enumerate}

{\it Proof of {\bf A2.}} \, Assume to the contrary that there is such a $\nu \in [1, t]$, say $\nu = 1$, with  $V_1 \t U A_1 \cdot \ldots \cdot A_{m-1}$. Then there are $l \in \N$ and $T_1, \ldots, T_l \in \mathcal A (G_0)$ such that
\[
U A_1 \cdot \ldots \cdot A_{m-1} = V_1 T_1 \cdot \ldots \cdot T_{\mathit l} \,.
\]
By the minimality of $m$,  there exists some $\nu \in [1, l]$ such that $T_{\nu} \in \Omega_{>1}$, say $\nu=1$.  Since
\[
\sum_{\nu=2}^l \mathsf k (T_{\nu}) = \mathsf k (U) + (m-1)-1 - \mathsf k (T_1) \le m-2 \le \frac{n-3}{2}\,,
\]
and $\mathsf k (T')\ge\frac{ r+n}{2}$ for all $T' \in \Omega_{>1}$, it follows that  $T_2, \ldots, T_l \in \Omega_{=1}$, whence $l = 1 + \sum_{\nu=2}^l \mathsf k (T_{\nu}) \le m-1$.
We obtain that
\[
V_1 T_1 \cdot \ldots \cdot T_{\mathit l}A_m = U A_1 \cdot \ldots \cdot A_m = V_1 \cdot \ldots \cdot V_t \,,
\] and thus
 \[
 T_1 \cdot \ldots \cdot T_{\mathit l}A_m = V_2 \cdot \ldots \cdot V_t \,.
\]
The minimality of $m$ implies that $\mathsf k(T_1)> \mathsf k(U)$.
It follows that  \[\mathsf k(T_1)-\mathsf k(U)=m-1-{\mathit l}\le m-2\le \frac{n-3}{2}<\frac{r+n}{2}-1\le \mathsf k(T_1)-\mathsf k(U)\,,\]
 a contradiction.  \qed{(Proof of {\bf  A2})}

\medskip
By Equation (\ref{e3}), there are $X_1, Y_1, \ldots,  X_t, Y_t \in \mathcal F (G)$ such that
\[
UA_1\cdot \ldots \cdot A_{m-1}=X_1\cdot\ldots\cdot X_t, \ A_{m}=Y_1\cdot\ldots\cdot Y_t, \ \text{and} \ V_i=X_iY_i \ \text{ for each} \ i\in[1,t] \,.
\]
Then {\bf A2} implies that $|Y_i| \ge 1$ for each $i \in [1,t]$, and we set
$\alpha=|\{i\in [1,t]\mid |Y_i|=1\}|$. If $\alpha\le m+ r$, then
\[
n\ge |A_m|=|Y_1|+\ldots+|Y_t|\ge \alpha+2(t-\alpha)=2t-\alpha\ge 2t-m- r \,,
\]
 and hence $\min \Delta (G_0)\le t-1-m\le \frac{r+n-3}{2}$, a contradiction. Thus
$\alpha\ge m+ r+1$.
After renumbering if necessary we  assume that $1=|Y_1|=\ldots=|Y_{\alpha}|<|Y_{\alpha+1}|\le \ldots\le |Y_t|$.
Let $Y_i=y_i$ for each $i\in[1,\alpha]$ and
\begin{equation}\label{e6}
S_0=\{y_1,y_2,\ldots,y_{\alpha}\}\,.
\end{equation}
For every $i \in [1, \alpha]$, $V_i \t y_iU A_1 \cdot \ldots \cdot A_{m-1}$ whence $\mathsf v_{y_i} (V_i) \le 1 + \mathsf v_{y_i} (U A_1 \cdot \ldots \cdot A_{m-1})$ and since $V_i \nmid U A_1 \cdot \ldots \cdot A_{m-1}$, it follows that
\begin{equation} \label{e7}
\mathsf v_{y_i}(V_i)=\mathsf v_{y_i}(UA_1\cdot\ldots\cdot A_{m-1})+1 \,.
\end{equation}
Assume to the contrary that there are distinct $i,j \in [1, \alpha]$ such that $y_i=y_j$. Then
\[
\mathsf v_{y_i}(UA_1\ \cdot \ldots \cdot  A_{m-1})+1=\mathsf v_{y_i}(V_i)=\mathsf v_{y_i}(X_i)+1 =\mathsf v_{y_i}(V_j)=\mathsf v_{y_i}(X_j)+1 \,.
\]
Since $X_iX_j \t UA_1\ \cdot \ldots \cdot  A_{m-1}$, we infer that
\[
\mathsf v_{y_i}(UA_1\ldots A_{m-1}) \ge \mathsf v_{y_i}(X_iX_j)=\mathsf v_{y_i}(V_iV_j)-2 = 2\mathsf v_{y_i}(UA_1\ldots A_{m-1}) \,,
\]
which implies that  $\mathsf v_{y_i}(UA_1\ldots A_{m-1})=0$, a contradiction to  $\supp(U)=G_0$. Thus $|S_0|=\alpha$ and

\begin{equation}\label{e5}
|\supp(A_m)|\ge |S_0|=\alpha\ge m+ r+1\,.
\end{equation}

 We proceed by the following two assertions.

\begin{enumerate}
\item[{\bf A3. }]  There exist   $g'\in G_0$ and $A' \in \mathcal A (G_0)$ with $\mathsf k(A')=1$ satisfying the following three conditions:
                   \begin{enumerate}
                   \item[{(C1) }] \  $\mathsf v_{g'}(A') < \ord (g')$ is the smallest positive integer $\gamma$ such that  $\gamma g'\in \langle\supp(A')\setminus\{g'\}\rangle$;
                   \item[{(C2) }] \  $\mathsf v_{g'}(A')g'\notin \langle E\rangle$ for any $E\subsetneq \supp(A')\setminus\{g'\})$.
                   \item[{(C3) }] \  $UA_1\cdot\ldots\cdot A_{m-1}\cdot A'$ can be factorized into a product of atoms from $\Omega_{=1}$.
                   \end{enumerate}
\end{enumerate}

{\it Proof of {\bf A3.}} \, Suppose that  Property (a) is  satisfied. As observed at the beginning of the proof, there is a  $g\in G_0$ such that $\mathsf v_g(A)=1$. We choose $A'=A$ and $g'=g$, and we  need to prove that $UA_1\cdot\ldots\cdot A_{m-1}\cdot A$ can be factorized into a product of atoms from $\Omega_{=1}$. Since $S_0\subset \supp(A)=G_0$, then $V_1\cdot\ldots\cdot V_{\alpha}\t UA_1\cdot\ldots\cdot A_{m-1}\cdot A$ and hence $\mathsf k(UA_1\cdot\ldots\cdot A_{m-1}\cdot A(V_1\cdot\ldots\cdot V_{\alpha} )^{-1})<\mathsf k(U)$. The minimality of $\mathsf k(U)$ implies that $UA_1\cdot\ldots\cdot A_{m-1}\cdot A$ can be factorized into a product of atoms from $\Omega_{=1}$.

Suppose that  Property (b) satisfied. We choose $g'=y_1$(see  Equation \eqref{e6}) and distinguish two cases.
First, suppose that there exists a subset $E\subsetneq G_0\setminus\{y_1\}$ such that $y_1\in \langle E\rangle$. Choose a minimal subset $E$ with this property.  By Lemma \ref{3.1}.1,  there exists an atom $A'$ satisfying the two conditions (C1) and (C2) with $\mathsf k(A')=1$ and $\mathsf v_{y_1}(A')=1$.
Since $\mathsf v_{y_1}(V_1)=\mathsf v_{y_1}(UA_1\cdot\ldots\cdot A_{m-1})+1$ by Equation \ref{e7} and $V_1\t UA_1\cdot\ldots\cdot A_{m-1}\cdot y_1$, we obtain that $|\supp(UA_1\cdot\ldots\cdot A_{m-1}\cdot A'(V_1)^{-1})|<|G_0|$ and hence $UA_1\cdot\ldots\cdot A_{m-1}\cdot A'$ can be factorized into a product of atoms from $\Omega_{=1}$.

Now we suppose  that $y_1\notin \langle E\rangle$ for any $E\subsetneq G_0\setminus\{y_1\}$.
Let $s_0 \in \N$  be minimal such that there exists a subset $E\subsetneq G_0\setminus\{y_1\}$ such that $s_0 y_1\in \langle E\rangle$, and  by Lemma \ref{3.3}.1, we observe that $s_0 < \ord (g)$.
Let $E$  be a minimal subset with this property. Thus, by Lemma \ref{3.1}.1,  there exists an atom $A'$ with $\mathsf v_{y_1}(A')=s_0$ satisfying the two conditions (C1) and (C2).  Since $\supp (A') \subsetneq G_0$, we have $\mathsf k(A')=1$. We distinguish two cases:

\medskip
\noindent
CASE 1: \ $ |S_0\setminus \supp(A')|\ge r+1$.

Since $s_0 \ge 2$, there is a  prime $p$ dividing $s_0$.
Since by assumption, $y_1 \in \langle G_0 \setminus \{y_1\} \rangle$,
Lemma \ref{3.3}  implies that  for each prime $p$ dividing $\ord(y_1)$, there exists an atom $A_p'$ such that $|\supp(A_p')|\le  r+1<|G_0|$, $1 \le \mathsf v_{y_1} (A_p') \le \ord (y_1)/2$, and $p\nmid \mathsf v_{y_1}(A_p')$.

Let $d=\gcd(s_0,\mathsf v_{y_1}(A_{p}'))$. Then $d<s_0$ and $dy_1\in \langle s_0 y_1, \mathsf v_{y_1}(A_{p}')y_1\rangle\subset \langle (\supp(A')\cup \supp(A_{p}')) \setminus \{y_1\}\rangle$. By the minimality of $s_0$, we have $G_0\setminus\{y_1\}=(\supp(A')\cup \supp(A_{p}'))\setminus \{y_1\}$. It follows that
\begin{align*}
|\supp(A')|+r&\ge|\supp(A')|+|\supp(A_{p}')|-1\ge |G_0|\\
&\ge |\supp(A')|+|S_0\setminus\supp(A')|\ge |\supp(A')|+r+1\,,
\end{align*}
 a contradiction.

\medskip
\noindent
CASE 2: \ $ |S_0\setminus \supp(A')|\le r$.
\medskip

Therefore $|\supp(A')\cap S_0|\ge m+1$ by Equation \eqref{e5}, and we may suppose that $\{y_1,\ldots,y_{m+1}\}\subset \supp(A')\cap S_0$.
Then $V_1\cdot\ldots\cdot V_{m+1}\t UA_1\cdot\ldots\cdot A_{m-1}A'$ and $\mathsf k (UA_1\cdot\ldots\cdot A_{m-1}A'(V_1\cdot\ldots\cdot V_{m+1})^{-1})<\mathsf k(U)$. By the minimality of $\mathsf k(U)$, we have that $UA_1\cdot\ldots\cdot A_{m-1}A'$ can be factorized into a product of atoms from $\Omega_{=1}$.
\qed{(Proof of {\bf A3})}

\medskip
\begin{enumerate}
\item[{\bf A4. }]  Let $g'\in G_0$ and $A' \in \mathcal A (G_0)$ with $\mathsf k(A')=1$ satisfying the following three conditions:
                   \begin{enumerate}
                   \item[{(C1) }] \  $\mathsf v_{g'}(A') < \ord (g')$ is the smallest positive integer $\gamma$ such that  $\gamma g'\in \langle\supp(A')\setminus\{g'\}\rangle$;
                   \item[{(C2) }] \  $\mathsf v_{g'}(A')g'\notin \langle E\rangle$ for any $E\subsetneq \supp(A')\setminus\{g'\})$.
                   \item[{(C3) }] \  $UA_1\cdot\ldots\cdot A_{m-1}\cdot A'$ can be factorized into a product of atoms from $\Omega_{=1}$.
                   \end{enumerate}
                   If  $|\supp(A')| \ge m+ r+1$, then there exists an atom $A'' \in \mathcal A (G_0)$ with $\mathsf k (A'')=1$ and $|\supp(A'')|<|\supp(A')|$  such that (C1), (C2), and (C3) hold.
\end{enumerate}

Suppose that {\bf A4} holds.
Iterating {\bf A4} we find an atom $A^*$ with $|\supp(A^*)|\le m+ r$ such that $UA_1\cdot\ldots\cdot A_{m-1}\cdot A^*$ can be factorized into a product of atoms from $\Omega_{=1}$, a contradiction to  \eqref{e5}.

\medskip
{\it Proof of {\bf A4.}} \,  For simplicity of notation, we suppose that  $A'=A_m$.

Let $s_0 \in \N$  be minimal such that there exists a subset $E \subsetneq \supp(A_m)\setminus\{g'\}$ such that $s_0 g'\in \langle E\rangle$. By (C1) and $|\supp (A')| \ge m+r+1\ge r+2$, Lemma \ref{3.3} implies that $s_0 < \ord (g')$.
Let $E$ be a minimal subset with this property.
Thus, by Lemma \ref{3.1}.1,  there exists an atom $A''$ with $\mathsf v_{g'} (A'')=s_0$ satisfying the two conditions (C1) and (C2).
Since $\supp (A'') \subsetneq G_0$, we have $\mathsf k (A'')=1$.
We distinguish two cases:

\medskip
\noindent
CASE 1: $ |S_0\setminus \supp(A'')|\ge r+1$.
\medskip

We set $s'=\mathsf v_{g'}(A_m) < \ord (g')$.
 Since $A_m$ satisfies   condition (C1), Lemma \ref{3.1}.1 implies  that $s'\t s_0$ and $\frac{s_0}{s'}>1$. Let $p$ be a prime dividing $\frac{s_0}{s'}$. Since $s' \t \ord (g')$ and $s_0 \t \ord (g')$, it follows that $p \t \frac{s_0}{s'} \t \frac{\ord (g')}{s'} =  \ord (s' g')$.
 Lemma \ref{3.3} (applied to the subset $\supp (A_m) \subset G$) implies  that  there exists an atom $A_p'\in \mathcal A(\supp(A_m))$ such that $|\supp(A_p')|\le r+1<|\supp(A_m)|$, $s' \le \mathsf v_{g'} (A_p') \le \ord (g')/2$, and $p\nmid \frac{\mathsf v_{g'}(A_p')}{s'}$.

Let $d=\gcd(\frac{s_0}{s'},\frac{\mathsf v_{g'}(A_{p}')}{s'})$. Then $d<\frac{s_0}{s'}$ and
\[
ds'g'\in \langle s_0 g', \mathsf v_{g'}(A_{p}')g'\rangle\subset \langle (\supp(A'')\cup \supp(A_{p}') ) \setminus \{g'\}\rangle \,.
\]
Thus by minimality of $s_0$, we have $\supp(A_m)\setminus\{g'\}= (\supp(A'')\cup \supp(A_{p}') )\setminus \{g'\}$.
It follows that
\begin{align*}
|\supp(A'')|+ r&\ge|\supp(A'')|+|\supp(A_{p}')|-1\ge |\supp(A_m)|\\
&\ge |\supp(A'')|+|S_0\setminus\supp(A'')|\ge |\supp(A'')|+ r+1\,,
\end{align*}
 a contradiction.

\medskip
\noindent
CASE 2: \, $ |S_0\setminus \supp(A'')|\le r$.
\medskip

 Therefore $|\supp(A'')\cap S_0|\ge m+1$ by Equation \eqref{e5}, and we may suppose  that $\{y_1,\ldots,y_{m+1}\}\subset \supp(A'')\cap S_0$.
Then $V_1\cdot\ldots\cdot V_{m+1}\t UA_1\cdot\ldots\cdot A_{m-1}A''$ and $\mathsf k (UA_1\cdot\ldots\cdot A_{m-1}A''(V_1\cdot\ldots\cdot V_{m+1})^{-1})<\mathsf k(U)$. By the minimality of $\mathsf k(U)$, we have that $UA_1\cdot\ldots\cdot A_{m-1}A''$ can be factorized into a product of atoms from $\Omega_{=1}$.  This completes the proof of ({\bf A4}) and thus Lemma \ref{3.6} is proved.
\end{proof}

\medskip
\begin{proposition} \label{3.7}
We have $\mathsf m(G) \le \min\{\frac{n}{2}+r-2,\max\{r-1, \frac{5n}{6}-4, \frac{n+r-3}{2}\}\}$.
\end{proposition}

\begin{proof}
Let $G_0 \subset G$ be a non-half-factorial LCN set. We have to prove that
\[
\min \Delta(G_0)\le \min\{\frac{n}{2}+r-2,\max\{r-1, \frac{5n}{6}-4, \frac{n+r-3}{2}\}\}\,.
\]
If $G_1 \subset G_0$ is non-half-factorial, then $\min \Delta (G_0) = \gcd \Delta (G_0) \t \gcd \Delta (G_1) = \min \Delta (G_1)$. Thus we may suppose that $G_0$ is minimal non-half-factorial. By Lemma \ref{3.1}.3.(a), we may suppose that $g \in \langle G_0 \setminus \{g\} \rangle$ for all $g \in G_0$.

If $n$ is a prime power, then $\mathsf m(G)=r-1$ by Proposition \ref{2.2}, and the assertion follows. Suppose that $n$ is not a prime power.
If $|G_0| \le r+1$, then $\min \Delta (G_0) \le |G_0|-2\le r-1$ by Lemma \ref{3.2}.3. Thus we may suppose that  $|G_0| \ge r+2$ and we distinguish two cases.

\smallskip
\noindent
CASE 1: \, There exists a subset $G_2 \subset G_0$ such that $\langle G_2 \rangle = \langle G_0 \rangle$ and $|G_2| \le |G_0|-2$.

Then Lemma \ref{3.6} implies that $\min \Delta (G_0)\le \frac{n+r-3}{2}$.

\smallskip
\noindent
CASE 2: \, Every subset $G_1\subset G_0$ with $|G_1|=|G_0|-1$ is a minimal generating set of $\langle G_0 \rangle$.

Then for each $h \in G_0$, $G_0\setminus \{h\}$ is half-factorial and $h \notin \langle G_0 \setminus \{h, h' \} \rangle \ \text{for any } \ h'\in G_0\setminus\{h\}$. Thus Lemma \ref{3.4} and Lemma \ref{3.6} imply that $\min \Delta(G_0)\le \max\{\frac{5n}{6}-4, \frac{n+r-3}{2}\}$. By Lemma \ref{3.4}, we obtain that $|G_0|\le r+\frac{n}{2}$. Therefore Lemma \ref{3.2}.3 implies that $\min \Delta(G_0)\le \min \{ r+\frac{n}{2}-2, \max \{ \frac{5n}{6}-4, \frac{n+r-3}{2} \}  \}$.
\end{proof}

\medskip
\begin{proposition} \label{3.8}
Let $G'$ be a finite abelian group with $\mathcal L(G)=\mathcal L(G')$.  If $r \in [2, (n-2)/4]$,  then $n=\exp(G')> \mathsf r(G')+1$.
\end{proposition}

\begin{proof}
Let $k\in \N$ be maximal such that $G$ has a subgroup isomorphic to $C_{n}^k$. Then $k\le r\le \frac{n-2}{4}\le \frac{n-3}{2}$. By Proposition \ref{3.7}, we obtain that
\[
\mathsf m(G)\le \frac{n}{2}+r-2\le n-2 \le n-k-3 \,
\]
and hence
\[
\max \{ \mathsf m(G), \lfloor \frac{n}{2}\rfloor-1 \} \le n-k-3\,.
\]
By Proposition \ref{2.2}.3,  we have
\[
\Delta_1(G)\subset [1,\max \{ \mathsf m(G), \lfloor \frac{n}{2}\rfloor-1 \}] \cup [n-k-1,n-2]\,,
\]
and thus $n-k-2\not\in \Delta_1(G)$. Thus  $n-k-2\not\in \Delta_1(G')$ and Proposition \ref{2.2} implies that
\[
n-2= \max \{ \mathsf m(G), n-2 \} = \max \Delta_1 (G) = \max \Delta_1 (G') = \max \{\mathsf r(G')-1, \exp(G')-2 \} \,.
\]
If $n-2=\mathsf r(G')-1$, then $\Delta_1(G')=[1,n-2]$ by Lemma \ref{3.2}.2, a contradiction to $n-k-2\not\in \Delta_1(G')$.
Therefore it follows that  $n=\exp(G')>\mathsf r(G')+1$.
\end{proof}

\medskip
\section{Proof of the Main Result and groups with small Davenport constant} \label{4}
\medskip

\medskip
\begin{proof}[Proof of Theorem \ref{1.1}]
Let $G$ be an abelian group such that $\mathcal L (G) = \mathcal L ( C_n^r)$ where $r, n \in \N$ with $n \ge 2$, $(n,r)\notin \{(2,1),(2,2), (3,1)\}$, and $r \le \max \{2, (n+2)/6 \}$.

First we note that $G$ has to be finite and that $\mathsf D(C_n^r)=\mathsf D(G)$ and $\Delta_1(C_n^r)=\Delta_1(G)$ (see \cite[Proposition 7.3.1 and Theorem 7.4.1]{Ge-HK06a}).
If $r=1$, then the assertion follows from \cite[Theorem 7.3.3]{Ge-HK06a}. If $r = 2$, then the assertion follows from \cite{Sc09c}, and hence we may suppose that $r \in [3, (n+2)/6]$.

Let $k\in \N$ be maximal such that $G$ has a subgroup isomorphic to $C_{n}^k$.  If $k \ge r$, then  $\mathsf D(C_n^r)=\mathsf D(G) \ge \mathsf D (C_n^k)$ implies that $k=r$ and that $G\cong C_n^r$. Suppose that $k <r$.
By Proposition \ref{3.8}, we obtain that $n=\exp(G)>\mathsf r(G)+1$.
By Proposition \ref{2.2}.3 (applied to $C_n^r$) we infer  that $[n-r-1,n-2]\subset \Delta_1(C_n^r) = \Delta_1(G)$.
 By Proposition \ref{2.2}.3 (applied to $G$), we obtain that
 \[
 [1,\mathsf r(G)-1] \cup  [n-r-1,n-2]\subset \Delta_1(G)\subset \big[1,\max \big\{\mathsf m(G), \lfloor \frac{n}{2}\rfloor-1 \big\} \big] \cup [n-k-1,n-2]\,,
 \]
which implies  that $\mathsf m(G)\ge n-r-1$. By Proposition \ref{3.7}, we have that
\[
n-r-1\le \mathsf m(G) \le \max \Big\{\mathsf r(G)-1, \frac{5n}{6}-4,\frac{n+\mathsf r(G)-3}{2} \Big\}\,.
\]

If $n-r-1\le \frac{5n}{6}-4$, then $r\ge \frac{n}{6}+3$, a contradiction. Thus $n-r-1\le \max \big\{\mathsf r(G)-1, \frac{n+\mathsf r(G)-3}{2} \big\}$ which implies that $\mathsf r(G)\ge n-2r+1$. Therefore
\begin{equation}\label{eq1}
[1,n-2r]\subset [1,\mathsf r(G)-1]\subset \Delta_1(G)=\Delta_1(C_n^r)\,.
\end{equation}
Since $r\le \frac{n+2}{6}$, we have that $n-2r-1\ge \frac{n}{2}+r-2\ge \lfloor \frac{n}{2}\rfloor-1$. By Proposition \ref{3.7}, we obtain that $\mathsf m(C_n^r)\le \frac{n}{2}+r-2\le n-2r-1$. Therefore
\[
\max \{ \mathsf m(C_n^r), \lfloor \frac{n}{2}\rfloor-1 \} < n-2r < n-r-1\,.
\]
 By Proposition \ref{2.2}.3, $n-2r\notin \Delta_1 (C_n^r)$, a contradiction to Equation \eqref{eq1}.
\end{proof}

\smallskip
Our proof of Theorem \ref{1.1}, to characterize the groups $C_n^r$ with $r, n$ as above, uses only the Davenport constant and the set of minimal distances. Clearly, there are non-isomorphic groups $G$ and $G'$ with $\mathsf D (G)=\mathsf D (G')$, $\Delta^* (G)=\Delta^* (G')$, and $\Delta_1 (G) = \Delta_1 (G')$. We meet this phenomenon in Proposition \ref{4.1}. Indeed, since $\mathcal L (C_1)=\mathcal L (C_2)$ and $\mathcal L (C_3)=\mathcal L (C_2 \oplus C_2)$ (\cite[Theorem 7.3.2]{Ge-HK06a}), small groups definitely deserve a special attention when studying the Characterization Problem. Clearly, the groups $C_1, C_2, C_3$, and $C_2 \oplus C_2$ are precisely the groups $G$ with $\mathsf D (G) \le 3$. In our final result we show that for all groups $G$ with $\mathsf D (G) \in [4,11]$ the  answer to the Characterization Problem is positive.

Suppose that $G \cong C_{n_1} \oplus \ldots \oplus C_{n_r}$ where $r \in \N_0$ and $1 < n_1 \t \ldots \t n_r$ and  set $\mathsf D^* (G) = 1 + \sum_{i=1}^r (n_i-1)$. Then $\mathsf D^* (G) \le \mathsf D (G)$. If   $\mathsf r (G)=r \le 2$ or if $G$ is a $p$-group, then equality holds.

\medskip
\begin{proposition} \label{4.1}
Let $G$ be a finite abelian group with $\mathsf D (G) \in [4,11]$. If $G'$ is a finite abelian group with $\mathcal L (G) = \mathcal L (G')$, then $G \cong G'$.
\end{proposition}

\begin{proof}
Suppose that $G'$ is a finite abelian group with $\mathcal L (G)=\mathcal L (G')$. Then $\mathsf D (G) = \mathsf D (G')$. If $\mathsf D (G) \in [4,10]$, then the assertion follows from \cite[Theorem 6.2]{Sc09c}.

Suppose that $\mathsf D (G)=\mathsf D (G')=11$. If $\mathsf r (G) \le 2$ or $\mathsf r (G') \le 2$, then the assertion follows from \cite[Theorem 1.1]{Ge-Sc16a}. If $G$ or $G'$ is an elementary $2$-group, then the assertion follows from \cite[Theorem 7.3.3]{Ge-HK06a}.

Thus we suppose that $\mathsf r (G) \ge 3$, $\mathsf r (G') \ge 3$,  $\exp (G) \in [3,8]$, and $\exp (G') \in [3, 8]$. If $G \cong C_4^3$ or $G' \cong C_4^3$, then the assertion follows from \cite[Theorem 4.1]{Ge-Sc16a}.  Thus we may suppose that all this is not the case. Since there is no finite abelian group $H$ with $\mathsf D (H)=11$ and $\exp (G) \in \{5,7\}$, it remains to consider the following groups:
\[
C_2^4 \oplus C_4^2, \  C_2^7 \oplus C_4, \
 \   C_2^r \oplus C_6, \  C_2^3 \oplus C_8, \ C_3^5 \quad \text{for some} \ r \in \N \,.
\]
By \cite[Corollary 2]{Ge-Sc92}, $\mathsf D (C_2^r \oplus C_6)= \mathsf D^* (G)=r+6$ if and only if $r \le 3$. Thus $\mathsf D (C_2^4 \oplus C_6) \ge 11$, and this is the only group for which $\mathsf D (C_2^r \oplus C_6)=11$ is possible. Thus we have to consider
\[
G_1= C_2^4 \oplus C_4^2, \  G_2=C_2^7 \oplus C_4, \
 \  G_4= C_2^4 \oplus C_6, \  G_5= C_2^3 \oplus C_8, \ \text{and} \ G_6=C_3^5  \,.
\]
Since $\max \Delta^* (G_1)=5$, $\max \Delta^* (G_2)=7$,  $\max \Delta^* (G_4)=4$,  $\max \Delta^* (G_5)=6$, and $\max \Delta^* (G_6)= 4$, it remains to show that $\mathcal L (C_2^4 \oplus C_6) \ne \mathcal L (C_3^5)$. Note that Proposition \ref{2.2} implies that $\Delta^* (C_2^4 \oplus C_6) = [1,4] = \Delta^* (C_3^5)$. By \cite[Theorem 6.6.2]{Ge-HK06a}, it follows that $\{2, 8\} \in \mathcal L (C_2^4 \oplus C_6)$, and we assert that $\{2,8\} \notin \mathcal L (C_3^5)$.

Assume to the contrary that $\{2,8\}\in \mathcal L(C_3^5)$. Then there exists $U,V\in \mathcal A(C_3^5)$ such that $\mathsf L(UV)=\{2,8\}$. We choose the pair $(U,V)$ such that $|U|$ is maximal and observe that $11 \ge |U| \ge |V| \ge 8$.
There exists an element $g \in G$ such that $g\t U$ and $-g\t V$. Then $\mathsf v_g(U)\le 2$ and $\mathsf v_{-g}(V)\le 2$. If $\mathsf v_g(U)=\mathsf v_{-g}(V)=2$, then $gV(-g)^{-2}$, $(-g)Ug^{-2}$ and $g(-g)$ are atoms and hence $3\in \mathsf L(UV)$, a contradiction. Therefore $\mathsf v_g(U)+\mathsf v_{-g}(V) \in [2, 3]$ and we set
\begin{equation}\label{eq2}
\{g\in \supp(U)\mid \mathsf v_g(U)+\mathsf v_{-g}(V)= 3\}=\{g_1,\ldots,g_{\ell}\} \quad \text{ where} \quad \ell\in \N_0 \,.
\end{equation}
We continue with the following assertion.

\begin{enumerate}
\item[{\bf A1.}\,] For each $i\in [1,\ell]$ we have $\mathsf v_{g_i}(U)=2$.
\end{enumerate}

{\it Proof of \,{\bf A1}}.\,
Assume to the contrary that there is an $i\in [1,\ell]$ with $\mathsf v_{g_i}(U)=1$. Then $g_iV((-g_i)^2)^{-1}$ is an atom and $(-g_i)^2Ug_i^{-1}$ is an atom or a product of two atoms. Since $3\notin \mathsf L(UV)$, we obtain that $(-g_i)^2Ug_i^{-1}$ is an atom but $|(-g_i)^2Ug_i^{-1}|>|U|$, a contradiction to our choice of $|U|$.  \qed{(Proof of {\bf A1})}

Now we set
\[
U'=(-g_1) \cdot \ldots \cdot (-g_{\ell})U(g_1^2 \cdot \ldots \cdot g_{\ell}^2)^{-1} \quad \text{and} \quad V'=g_1^2 \cdot \ldots \cdot g_{\ell}^2 V ((-g_1) \cdot \ldots \cdot (-g_{\ell}))^{-1}\,.
\]
Using the above argument repeatedly we infer that  $U'$ and $V'$ are atoms. Clearly, we have $\mathsf L (U'V')=\mathsf L (UV) = \{2, 8\}$ whence $|V|+\ell= |V'|\le |U|$ and thus $\ell \le 3$.
We consider a factorization
\[
UV=W_1\cdot \ldots \cdot W_8\,,
\]
where $W_1, \ldots, W_8 \in \mathcal A(C_3^5)$  such that $|\{i\in [1,8]\mid |W_i|=2\}|$ is maximal under all factorization of $UV$ of length $8$. We set $U=U_1 \cdot \ldots \cdot U_8$,  $V=V_1 \cdot \ldots \cdot V_8$ such that $W_i=U_iV_i$ for each $i\in [1,8]$, and we define $W=\sigma(U_1) \cdot \ldots \cdot \sigma(U_8)$.
We continue with a second assertion.

\begin{enumerate}
\item[{\bf A2.}\,]
There exist disjoint non-empty subsets  $I,J,K\subset [1,8]$ such that
\[
I\cup J\cup K=[1,8] \quad \text{ and} \quad  \sigma \big( \prod_{i\in I}U_i \big) =\sigma \big( \prod_{j\in J}U_j \big) = \sigma \big( \prod_{k\in K}U_k \big) \,.
\]
\end{enumerate}

{\it Proof of \,{\bf A2}}.\, First we suppose that  $\mathsf h(W)\ge 2$, say $\sigma (U_1)=\sigma (U_2)$. Then $I=\{1\}$, $J= \{2\}$, and $K=[3,8]$ have the required properties.
Now suppose that $\mathsf h(W)=1$. Since the tuple $(\sigma(U_1),\ldots, \sigma(U_7))$ is not independent and the sequence $\sigma(U_1) \cdot \ldots \cdot \sigma(U_7)$ is zero-sum free,  there exist disjoint non-empty subset $I,J\subset [1,8]$, such that $\sum_{i\in I} \sigma (U_i) =\sum_{j\in J} \sigma (U_j)$. Therefore, $I$, $J$, and  $K=[1,8]\setminus (I\cup J)$ have the required properties.
\qed{(Proof of {\bf A2})}

We define
\[
X_1=\prod_{i\in I}U_i\,, X_2=\prod_{j\in J}U_j\,, \text{ and } X_3=\prod_{k\in K}U_k\,,
\]
\[
Y_1=\prod_{i\in I}V_i\,, Y_2=\prod_{j\in J}V_j\,, \text{ and } Y_3=\prod_{k\in K}V_k\,.
\]
By construction, we have $X_1Y_1=\prod_{i\in I} W_i$, $X_2Y_2=\prod_{j \in J} W_j$, $X_3Y_3=\prod_{k \in K} W_k$, $\sigma (X_1)=\sigma (X_2)=\sigma (X_3)$, $\sigma (Y_1)=\sigma (Y_2)=\sigma (Y_3)$, and hence $X_iY_j \in \mathcal B (G)$ for all $i, j \in [1,3]$.

We choose a factorization of $X_1Y_2$, a factorization of $X_2Y_3$, and a factorization of $X_3Y_1$, and their product gives rise to a factorization of $UV$, say $UV = W_1' \cdot \ldots \cdot W_8'$, where all the $W_i'$ are atoms, and we denote by $t_1$ the number of $W_i'$ having length two.
Similarly, we choose a factorization of
$X_1Y_3$, a factorization of $X_2Y_1$, and a factorization of $X_3Y_2$, obtain a factorization of  $UV$, and we  denote by $t_2$  the number of atoms of length $2$ in this  factorization.
If $g \in G$ and $i, j \in [1,3]$ are distinct with $g(-g) \t X_i Y_j$ and $g \t X_i$, then the choice of the factorization $UV = W_1 \cdot \ldots \cdot W_8$ implies that $g(-g) \t X_i Y_i$ or $g(-g) \t X_j Y_j$ whence $\mathsf v_g (U)+\mathsf v_{-g}(V)\ge 3$. Therefore Equation (\ref{eq2}) implies  that $g \in \{g_1, \ldots, g_{\ell} \}$ whence $t_1+t_2 \le \ell \le 3$, and we may suppose that $t_1 \le 1$. Therefore we infer that
\[
2+ 3 \times 7 \le \sum_{i=1}^8 |W_i'| = |UV| \le 2 \mathsf D (C_3^5) = 22 \,,
\]
a contradiction.
\end{proof}
%%%%%%%%%%%%%%%%%%%%%%%%%%%%%%%%%%%%%%%%%%%%%%%%%%%%%%%%%%%%%%%%%%%%%%%%%
%% BIBLIOGRAPHY  %%%%%%%%%%%%%%%%%%%%%%%%%%%%%%%%%%%%%%%%%%%%%%%%%%%%%%%%
%%%%%%%%%%%%%%%%%%%%%%%%%%%%%%%%%%%%%%%%%%%%%%%%%%%%%%%%%%%%%%%%%%%%%%%%%

\providecommand{\bysame}{\leavevmode\hbox to3em{\hrulefill}\thinspace}
\providecommand{\MR}{\relax\ifhmode\unskip\space\fi MR }
\providecommand{\MRhref}[2]{%
  \href{http://www.ams.org/mathscinet-getitem?mr=#1}{#2}
}
\providecommand{\href}[2]{#2}

\end{document}